\documentclass[11pt,letterpaper]{article}

\usepackage{mypreprint2010}

\usepackage{amsmath}
\usepackage{amssymb}
\usepackage{tabularx}

\usepackage{enumerate}
\usepackage{algorithm}
\usepackage{algpseudocode}

\usepackage{graphicx}
\usepackage{subfigure}

\def\bfA{\boldsymbol{A}} 
\def\bfB{\boldsymbol{B}} 
\def\bfC{\boldsymbol{C}} 
 
\def\bfE{\boldsymbol{E}} 
 
\def\bfG{\boldsymbol{G}} 
\def\bfH{\boldsymbol{H}} 
\def\bfI{\boldsymbol{I}} 
 
\def\bfK{\boldsymbol{K}} 
\def\bfL{\boldsymbol{L}} 
\def\bfM{\boldsymbol{M}}

\def\bfP{\boldsymbol{P}} 
 
\def\bfR{\boldsymbol{R}} 
\def\bfS{\boldsymbol{S}} 
\def\bfT{\boldsymbol{T}} 
 \def\bfu{\boldsymbol{u}}
\def\bfV{\boldsymbol{V}} 
\def\bfW{\boldsymbol{W}} 
\def\bfX{\boldsymbol{X}} \def\bfx{\boldsymbol{x}}
 \def\bfy{\boldsymbol{y}}
\def\bfZ{\boldsymbol{Z}} \def\bfz{\boldsymbol{z}}
\def\bfnull{\boldsymbol{0}}

\def\Ars{\boldsymbol{\widetilde{A}}_{q}}
\def\Brs{\boldsymbol{\widetilde{B}}_{q}}  \def\Bsn{\boldsymbol{\widetilde{B}}}
\def\Crs{\boldsymbol{\widetilde{C}}_{q}}  \def\Csn{\boldsymbol{\widetilde{C}}}
\def\Ers{\boldsymbol{\widetilde{E}}_{q}}
\def\Grs{\boldsymbol{\widetilde{G}}_{q}} 

\def\Bs{\boldsymbol{B}_{\bot}}          
\def\Bss{\boldsymbol{B}_{\bot}^T}
\def\Bsv{\boldsymbol{B}_{\bot,\bfE\bfV}}

\def\Ph{\boldsymbol{\widehat{P}}}

\def\H2{\mathcal{H}_2}
\def\Zs{\widetilde{\boldsymbol{Z}}} 
\def\Real{\boldsymbol{T}_R} 
\def\bfsigma{\boldsymbol{\sigma}}
\def\bflambda{\boldsymbol{\lambda}}

\DeclareMathOperator*{\trace}{trace}
\DeclareMathOperator*{\Rank}{rank}
\DeclareMathOperator*{\colsp}{span}

\begin{document}

\date{\today}

\title{The ADI iteration for Lyapunov equations implicitly performs $\mathcal{H}_2$ pseudo-optimal model order reduction}

\author{Thomas Wolf\footnotemark[2]\ \ and Heiko K. F. Panzer\footnotemark[2]\ \footnotemark[3]}

\maketitle

\renewcommand{\thefootnote}{\fnsymbol{footnote}}
\footnotetext[2]{Institute of Automatic Control,
				Technische Universit\"at M\"unchen,
				Boltzmannstr. 15, D-85748 Garching (\textbraceleft \texttt{thomas.wolf,panzer}\textbraceright\texttt{@tum.de})}
\footnotetext[3]{Partially supported by the Cusanuswerk Scholarship Award Programme, which is gratefully acknowledged.}
\renewcommand{\thefootnote}{\arabic{footnote}}

\begin{abstract}
Two approaches for approximating the solution of large-scale Lyapunov equations are considered: the \emph{alternating direction implicit} (ADI) iteration and projective methods by Krylov subspaces.
A link between them is presented by showing that the ADI iteration can always be identified by a Petrov-Galerkin projection with rational block Krylov subspaces.
Then a unique Krylov-projected dynamical system can be associated with the ADI iteration,
which is proven to be an $\H2$ pseudo-optimal approximation.
This includes the generalization of previous results on $\H2$ pseudo-optimality to the multivariable case.
Additionally, a low-rank formulation of the residual in the Lyapunov equation is presented, which is well-suited for implementation, and which yields a measure of the ``obliqueness'' that the ADI iteration is associated with.

\end{abstract}

\section{Introduction}
The generalized algebraic Lyapunov equation takes the form
\begin{equation}\label{eq:lyap}
	\bfA\bfP\bfE^T + \bfE\bfP\bfA^T +\bfB\bfB^T = \bfnull,
\end{equation}
with $\bfA,\bfE,\bfP \in \mathbb{R}^{n \times n}$ and $\bfB \in \mathbb{R}^{n \times m}$.
We assume $n$ to be large and $m$ small, $m \ll n$.
The matrix $\bfE$ can be arbitrary of full rank, det$(\bfE) \ne 0$, because in large-scale settings it is often reasonable to work with the general version \eqref{eq:lyap} of the Lyapunov equation.
The standard formulation results from setting $\bfE = \bfI$, where $\bfI$ denotes the identity matrix.
The Lyapunov equation \eqref{eq:lyap} plays an important role in the analysis and order reduction of linear time invariant dynamical systems
\begin{equation}\label{eq:sys}
\begin{aligned}
	\bfE \dot{\bfx}(t) &= \bfA\bfx(t) + \bfB \bfu(t), \\
	\bfy(t) &= \bfC \bfx(t),
\end{aligned}
\end{equation}
where $\bfx(t) \in \mathbb{R}^{n}$, $\bfu(t) \in \mathbb{R}^{m}$ and $\bfy(t) \in \mathbb{R}^{p}$ denote the states, inputs and outputs of the system, respectively, and $\bfC \in \mathbb{R}^{p \times n}$.
With the usual abuse of notation, let $\bfG(s) = \bfC \left( s\bfE - \bfA \right)^{-1}\bfB$ denote the transfer function of system \eqref{eq:sys} in the Laplace domain as well as the dynamical system itself.
The solution $\bfP$ of \eqref{eq:lyap} represents the \emph{Controllability Gramian} of $\bfG(s)$, see \cite{antoulas} for details.

Well-established methods for the direct (also called dense) solution of \eqref{eq:lyap} are available in matrix computation software, \cite{Bartels_Stewart,Hammarling,direct_lyap}.
However, for large $n$, their evaluation becomes time-consuming or might even fail due to shortage of RAM.
A remedy is to apply iterative methods that take advantage of the sparsity of the matrices $\bfA$, $\bfE$ and $\bfB$, to compute low-rank approximations $\Ph \approx \bfP$.
This low-rank formulation also allows the efficient storage of the approximation $\Ph$ even in large-scale settings.

A prevalent method for the approximate solution of \eqref{eq:lyap} is the \emph{alternating directions implicit} (ADI) iteration, which was adapted to large-scale systems by a low-rank formulation in \cite{Li_ADI,penzl_smith}.
A comprehensive analysis of ADI is available in the two theses \cite{diss_saak,diss_sabino}, whereas recent results can be found e.\,g. in \cite{paed_MCMDS,real_ADI,Saak_2012_MathMod}.

Another way for approximating $\bfP$ is to project \eqref{eq:lyap} to reasonable order $q \ll n$ and solve the resulting reduced Lyapunov equation by direct methods.
Krylov subspaces are typically used for this projection, see e.\,g. \cite{Jaimoukha_Krlyov_Lyap,Jbilou_Krylov_Lyap,newlya}.
If \emph{rational} Krylov subspaces are employed, this procedure is called \emph{rational Krylov subspace method} (RKSM) \cite{Druskin_2011}.
Several aspects of RKSM are analyzed in \cite{Beckermann_2011,Breiten_ADI,Druskin_Convergence,Wolf_2013_ACC}.
If particular sets of shifts are used for both ADI and RKSM, it can be shown that the resulting approximations $\Ph$ are equal \cite{Druskin_Convergence,flagg_ADI}.

The authors of \cite{Druskin_2011} restrict themselves to orthogonal projections;
however, oblique projections can be directly incorporated into the framework of RKSM, which introduces a new degree of freedom.
As we will show, this degree of freedom in RKSM can always be chosen such that the resulting approximation $\Ph$ is equal to the one of the ADI iteration---irrespective of the choice of shifts.

Accordingly, we present the missing link of both methods: we show how the ADI solution can be obtained by (oblique) projections with Krylov subspaces.
This generalizes the connection of ADI with RKSM from \cite{Druskin_Convergence,flagg_ADI} to arbitrary shifts.
Due to this link, a reduced order model by a Krylov-based projection of \eqref{eq:sys} can be associated to the ADI iteration.
We prove that this reduced system is unique and that it is an $\H2$ pseudo-optimal reduction of \eqref{eq:sys}.
Besides a better understanding of ADI, the new link allows to carry over results on Krylov-based projections.
As a first application, we adapt \cite{Wolf_2013_ACC} to ADI and present a new low-rank formulation of the residual (with maximum rank $m$), 
which significantly reduces the numerical effort in the evaluation of stopping criteria in the low-rank ADI iteration.
It furthermore allows to efficiently compute the angle between the oblique projection---that the ADI iteration is associated with---and the orthogonal projection---that RKSM usually performs.

\section{Preliminaries and problem formulation}
In this section relevant preliminaries are reviewed.
Throughout the paper we assume the system \eqref{eq:sys} to be asymptotically stable, which means that the set of eigenvalues of the matrix $\bfE^{-1}\bfA$ lies in the open left half of the complex plane.
Then the solution $\bfP$ of the Lyapunov equation \eqref{eq:lyap} is unique and positive semi-definite, $\bfP = \bfP^T \ge \bfnull$; it is positive definite, if in addition, system \eqref{eq:sys} is controllable, \cite{antoulas}.
Methods for the approximate solution of \eqref{eq:lyap} try to find a positive semi-definite matrix $\Ph \in \mathbb{R}^{n \times n}$ of rank $q$, with $q \ll n$, such that $\Ph \approx \bfP$.

\subsection{Alternating directions implicit (ADI) iteration}\label{sec:ADIrew}
In the  basic ADI iteration, the user chooses a sequence of complex shifts $(\sigma_1,\sigma_2,\ldots,\sigma_k )$ (here, the $\sigma_i \in \mathbb{C}$ are chosen in the right half of the complex plane)
and an initial approximation  $\Ph_0$ (e.\,g. $\Ph_0  =\bfnull$).
For the case $\bfE = \bfI$, the ADI approximation $\Ph$ is determined by the following iteration:
\begin{equation}\label{eq:basicADI}
	\begin{aligned} 
		\left( \bfA - \sigma_i\bfI \right) \Ph_{i-\frac{1}{2}} &= -\bfB\bfB^T -
		\Ph_{i-1}\left( \bfA^T - \sigma_i\bfI \right),  \\
		\left( \bfA - \sigma_i\bfI \right) \Ph_{i}^T &= -\bfB\bfB^T -
		\Ph_{i-\frac{1}{2}}\left( \bfA^T - \sigma_i\bfI \right).
	\end{aligned}
\end{equation}
Li et.\,al. observed in \cite{Li_ADI}, that for the choice $\Ph_0  =\bfnull$ the $k$-th iterate of \eqref{eq:basicADI} can be reformulated as a low-rank factor $\Ph = \bfZ\bfZ^*$, where ${}^*$ denotes transposition with complex conjugation.
With the generalization to arbitrary $\bfE \ne \bfI$, the ADI based solution factor $\bfZ = [\bfZ_1, \, \ldots, \, \bfZ_k ]$ is then given by:
\begin{equation}\label{eq:LRADI}
	\begin{aligned} 
		\bfZ_1 &= \sqrt{2\, \textrm{Re} (\sigma_1) }  \left( \bfA - \sigma_1\bfE \right)^{-1} \bfB, \\
		\bfZ_{i} &= \sqrt{\frac{ \textrm{Re} (\sigma_{i})}{ \textrm{Re} (\sigma_{i-1}) } }
		\left( \bfI + ( \sigma_{i} + \bar{\sigma}_{i-1} )    \left( \bfA - \sigma_{i}\bfE \right)^{-1}\bfE \right) \bfZ_{i-1}, \qquad\, i=2,\ldots,k.
	\end{aligned}
\end{equation}
For the ease of presentation, we assume that the set $\bfsigma = \{\sigma_1, \ldots, \sigma_k \} $ contains distinct shifts, $\sigma_i \ne \sigma_j$, $i\ne j$.
However, the results of this work are also valid without this assumption.
A matrix $\bfZ \in \mathbb{C}^{n \times q}$ constructed by the low-rank ADI iteration \eqref{eq:LRADI} will be referred to as the \emph{ADI basis} in the following; its column dimension is $q = k \cdot m$.
As the matrices $\bfE$, $\bfA$ and $\bfB$ are real, in ``almost every practical situation'' \cite{flagg_ADI} one would choose the set $\bfsigma$ to be closed under complex conjugation.
Then there has to exist a non-singular matrix $\Real \in \mathbb{C}^{q \times q}$ such that $\Zs=\bfZ\Real^{-1} \in \mathbb{R}^{n \times q}$ is a real ADI basis.
The Lyapunov approximation $\Ph$ then is given by $\bfZ\bfZ^* = \Zs\Real\Real^*\Zs^T$.
In \cite{real_ADI}, an analytical expression for $\Real$ is presented, which is used to slightly modify the iteration \eqref{eq:LRADI} for directly computing a real ADI basis.

For ADI (and also RKSM), the set $\bfsigma$ has to be chosen a priori or by iterative procedures. 
Several works are available on the choice of shifts; see e.\,g. \cite{penzl_smith,H2_gugercin,Druskin_2011,trac_eid_09} to mention just a few of them.

\subsection{Rational Krylov subspace method (RKSM)}\label{sec:Krylov}
It was shown in \cite{Li_ADI}, that the ADI basis $\bfZ$ spans the rational input block Krylov subspace $\mathcal{K}(\bfA,\bfB,\bfsigma)$ defined as:
\begin{equation}\label{eq:Krylov_s}
	\mathcal{K}(\bfA,\bfB,\bfsigma) := \colsp\left\{ (\bfA-\sigma_1\bfE)^{-1}\bfB, \, \ldots, \,(\bfA-\sigma_k\bfE)^{-1}\bfB \right\}.
\end{equation}
Let the matrix $\bfV \in \mathbb{C}^{n \times q}$ denote an arbitrary basis of this subspace:
$\colsp(\bfV) \!=\! \mathcal{K}(\bfA,\!\bfB,\!\bfsigma)$.
As the set $\bfsigma$ is assumed to be closed under complex conjugation, a real basis $\bfV \in \mathbb{R}^{n \times q}$ of the subspace \eqref{eq:Krylov_s} can be computed \cite{grimme}.
Due to numerical reasons, this basis $\bfV$ is usually chosen orthogonal: $\bfV^T\bfV = \bfI$, which is typically done by an Arnoldi approach \cite{antoulas}.
However, from a theoretical point of view, we do not require orthogonality here. 
Throughout the paper we shall assume that all directions in the rational block Krylov subspace \eqref{eq:Krylov_s} are linearly independent, which means that the dimension of the subspace is $q = k \cdot m$.

The basic idea of RKSM is to use projections onto Krylov subspaces $\bfV$.
Let $\bfW \in \mathbb{R}^{n \times q}$ be arbitrary of appropriate dimensions, then the reduced matrices from a \emph{Petrov--Galerkin} projection read as $\bfA_q = \bfW^T\bfA\bfV$, $\bfE_q = \bfW^T\bfE\bfV$ and $\bfB_q = \bfW^T\bfB$.
They define a reduced Lyapunov equation
\begin{equation}\label{eq:Pr}
	\bfA_q\bfP_q\bfE_q^T + \bfE_q\bfP_q\bfA_q^T +\bfB_q\bfB_q^T = \bfnull,
\end{equation}
which then can be cheaply solved by direct methods for $\bfP_q \in \mathbb{R}^{q \times q}$.
The approximation of RKSM then is given by $\Ph = \bfV\bfP_q\bfV^T$.

Please note, that RKSM was introduced in \cite{Druskin_2011} with a \emph{Galerkin} projection $\bfW := \bfV$.
However---as the basic procedure is left unchanged---we still refer to the generalized method $\bfW \ne \bfV$ as RKSM.

Also note that the approximation $\Ph$ is invariant to coordinate changes \cite{Wolf_2013_ACC}.
This means that only the column span of $\bfV$ affects the approximation $\Ph$ while the chosen basis is irrelevant.
Therefore, fixing the set $\bfsigma$ uniquely defines $\colsp(\bfV)$, and the only remaining degree of freedom in $\Ph$ from RKSM is the choice of $\bfW$.

\subsection{Contributions of this work}
We assume the set $\bfsigma$ arbitrary but fixed and apply it to both ADI and RKSM.
As shown in \cite{Li_ADI}, the ADI basis $\bfZ$ and the Krylov basis $\bfV$ then span the same subspace: $\colsp(\bfV) = \colsp(\bfZ)$.
Our first contribution is an alternative proof of this fact, which is based on rephrasing the ADI iteration \eqref{eq:LRADI} into a Sylvester equation.
This will be the starting point for our main contributions.

As  $\colsp(\bfV) = \colsp(\bfZ)$ holds, there exists a nonsingular matrix $\bfT \in \mathbb{C}^{q \times q}$ such that $\bfZ=\bfV\bfT$.
Comparing the ADI approximation $\Ph = \bfZ\bfZ^* = \bfV\bfT\bfT^*\bfV^T$ with the one of RKSM $\Ph = \bfV\bfP_q\bfV^T$, leads to the following interpretation:
if there exists a link between ADI and RKSM, then $\bfP_q := \bfT\bfT^*$ should solve a reduced Lyapunov equation \eqref{eq:Pr}.
Concerning the degrees of freedom in RKSM, this leads to the following question:

\emph{Is there a projection matrix $\bfW$ such that the resulting reduced Lyapunov equation \eqref{eq:Pr} is solved by $\bfT\bfT^*$; or in other words, is there a reduced Lyapunov equation \eqref{eq:Pr} that can be associated to the ADI iteration?}

We will confirm this in \S\ref{sec:ADIH2}, i.\,e.~prove the existence of $\bfW$,
for which ADI and RKSM yield the same approximation: $\Ph = \bfZ\bfZ^* = \bfV\bfP_q\bfV^T$.
Additionally, we prove that the associated reduced system matrices are unique and that they define an $\H2$ pseudo-optimal approximation of the original model $\bfG(s)$.

Due to this connection, the knowledge on Krylov-based projections can be transferred to ADI, which leads to our third contribution:
a numerically efficient computation and storage of the ADI residual and furthermore, a measure of the ``obliqueness'', i\,.e. the angle between the orthogonal and oblique projection.

\section{Analysis of ADI iteration and rational Krylov subspace method}\label{sec:main}
In this section the contributions of this work are presented.

\subsection{ADI basis spans a Krylov subspaces}
Our contributions require the following two lemmas, which give an alternative proof of $\colsp(\bfV) \!=\! \colsp(\bfZ)$.
First a new Sylvester equation is constructed, whose solution is the ADI basis $\bfZ$.
This reveals a new and alternative look on the ADI iteration \eqref{eq:LRADI} and facilitates its analysis.
\begin{lemma}\label{lem:ADI_Syl}
	Let $\alpha_i := \sqrt{2\,\textrm{Re}(\sigma_{i})}$, let $\bfI$ denote the identity matrix of dimension $m \times m$ and define
	\begin{equation}\label{eq:S_ADI}
		\bfS_{\text{ADI}} = \left[\begin{array}{cccc}
		\sigma_1\bfI & \alpha_1\alpha_2\bfI & \cdots & \alpha_1\alpha_l\bfI \\
		& \ddots & \ddots & \vdots \\ & & \ddots & \alpha_{l-1}\alpha_l\bfI \\ & & & \sigma_l\bfI \end{array} \right] \quad \text{and} \quad
		\bfL_{\text{ADI}} := \left[\alpha_1\bfI, \, \ldots,\, \alpha_k\bfI\right].
	\end{equation}
	Then the ADI basis $\bfZ$ from the iteration \eqref{eq:LRADI} solves the Sylvester equation
	\begin{equation}\label{eq:ADISyl}
		\bfA\bfZ - \bfE\bfZ \bfS_{\text{ADI}} = \bfB \bfL_{\text{ADI}}.
	\end{equation}
\end{lemma}
\begin{proof}
	From the Sylvester equation \eqref{eq:ADISyl} and the definitions \eqref{eq:S_ADI} it directly follows that
	\begin{equation}\label{eq:Z_iter}
		\left(\bfA - \sigma_i \bfE\right) \bfZ_i =
		\alpha_i \left( \sum_{j=1}^{i-1} \alpha_j\bfE\bfZ_j + \bfB \right).
	\end{equation}
	We prove the equivalence of the ADI iteration and \eqref{eq:Z_iter} by induction.
	Obviously, $\bfZ_1$ in \eqref{eq:Z_iter} is equal to the one of the ADI iteration \eqref{eq:LRADI}.
	Now assume that $\bfZ_i$ from \eqref{eq:LRADI} is given by \eqref{eq:Z_iter} and substitute $-\sigma_i = \bar{\sigma}_i + \alpha_i^2$. Then \eqref{eq:Z_iter} becomes
	\begin{equation}\label{eq:Z_iter2}
		\left(\bfA + \bar{\sigma}_i \bfE\right) \bfZ_i =
		\alpha_i \left( \sum_{j=1}^{i} \alpha_j\bfE\bfZ_j + \bfB \right).
	\end{equation}
	which is equivalent to
	\begin{equation}\label{eq:Z_iter3}
		\left(\bfA \!-\! \sigma_{i+1} \bfE\right)
		\left[ \bfI \!+\! (\sigma_{i+1} \!+\! \bar{\sigma}_i) \left(\bfA \!-\! \sigma_{i+1}\bfE \right)^{-1}\!\bfE \right] \! \bfZ_i \!=\! 
		\alpha_i \!\left( \sum_{j=1}^{i} \alpha_j\bfE\bfZ_j \!+\! \bfB \right).
	\end{equation}
	Using $\left[ \bfI \!+\! (\sigma_{i+1} \!+\! \bar{\sigma}_i) \left(\bfA \!-\! \sigma_{i+1}\bfE \right)^{-1}\!\bfE \right] \! \bfZ_i = \frac{\alpha_i}{\alpha_{i+1}}\bfZ_{i+1}$ from \eqref{eq:LRADI}, shows that \eqref{eq:Z_iter} is true for $\bfZ_{i+1}$, which completes the proof by induction.
	\hfill
\end{proof}

The result $\colsp(\bfV) \!=\! \colsp(\bfZ)$, originally given in \cite{Li_ADI}, now directly follows by the duality of Krylov subspaces and the solutions of Sylvester equations \cite{Gallivan_2004}.
In that sense, the following proof is simpler than the original one in \cite{Li_ADI}.
\begin{lemma}\label{lem:ADI_Krylov}
	The ADI basis $\bfZ$ from the iteration \eqref{eq:LRADI} spans a rational input block Krylov subspace $\mathcal{K}(\bfA,\bfB,\bfsigma)$.
\end{lemma}
\begin{proof}
	Lemma~\ref{lem:ADI_Syl} proves that $\bfZ$ spans a Krylov subspace and that the expansion points correspond to the eigenvalues of $\bfS_{\text{ADI}}$, \cite{Gallivan_2004,Wolf_2012_MathMod}.
	The eigenvalues of $\bfS_{\text{ADI}}$ directly follow from \eqref{eq:S_ADI}, which proves that $\bfZ$ spans the rational input block Krylov subspace $\mathcal{K}(\bfA,\bfB,\bfsigma)$.
	\hfill
\end{proof}

\subsection{Interpretation of ADI as a rational Krylov subspace method}\label{sec:ADI_RKSM}
The main contributions of this work build upon the following theorem.
\begin{theorem}[\cite{flagg_ADI}]\label{thm:gugercin}
	Given a set $\bfsigma = \{\sigma_1, \ldots, \sigma_k \} $ of distinct shifts that is closed under complex conjugation, let $\bfZ \in \mathbb{C}^{n \times q}$ be the basis resulting from the ADI iteration \eqref{eq:LRADI},
	and $\bfV \in \mathbb{R}^{n \times q}$ be an arbitrary real basis of the rational input block Krylov subspace $\mathcal{K}(\bfA,\bfB,\bfsigma)$.
	Let $\bfP_q \in \mathbb{R}^{q \times q}$ solve the projected Lyapunov equation
	\begin{equation}
		\bfA_q\bfP_q\bfE_q^T + \bfE_q\bfP_q\bfA_q^T +\bfB_q\bfB_q^T = \bfnull,
	\end{equation}
	with $\bfA_q := \bfV^T\bfA\bfV$, $\bfE_q := \bfV^T\bfE\bfV$ and $\bfB_q := \bfV^T\bfB$. 
	Then  $\Ph := \bfV\bfP_q\bfV^T = \bfZ\bfZ^*$, if and only if $\bfE_q^{-1}\bfA_q$ is diagonalizable with $k$ distinct eigenvalues $\bflambda = \{-\sigma_1, \ldots, -\sigma_k \}$, and each eigenvalue $\lambda_i = -\sigma_i$ has multiplicity $m$.
\end{theorem}

The theorem states that the approximations $\Ph$ of ADI and the orthogonal RKSM are equal, if and only if the eigenvalues of the projected matrix $\bfE_q^{-1}\bfA_q$ are the mirror images of the shifts $\bfsigma$, with respect to the imaginary axis.
Obviously, this condition is not true for arbitrary sets $\bfsigma$.
It can be fulfilled only for very particular sets $\bfsigma$.
However, such a set is previously unknown, and only an iterative procedure can be used to compute it.
Yet this iterative method has in general no guarantee to converge and is often numerically expansive, see \cite{flagg_ADI}.
This is why the theorem is mainly of theoretical interest and less of practical relevance.

Our aim is to give a constructive result, i.\,e. we want to show, how the eigenvalues can always be enforced at the mirror images of the shifts.
That means, we generalize the result of Theorem~\ref{thm:gugercin} to arbitrary sets $\bfsigma$.
Towards this aim, we use the following observation:
Although not explicitly stated in \cite{flagg_ADI}, Theorem~\ref{thm:gugercin} is still valid if the orthogonally projected matrices are substituted with obliquely projected ones, i.\,e. with $\bfA_q := \bfW^T\bfA\bfV$, $\bfE_q := \bfW^T\bfE\bfV$ and $\bfB_q := \bfW^T\bfB$, where $\bfW \in \mathbb{R}^{n \times q}$ is arbitrary.
The following theorem shows that with the additional degree of freedom $\bfW$, the condition in Theorem~\ref{thm:gugercin} can be achieved for arbitrary sets of shifts.
\begin{theorem}\label{thm:main}
	Given a set $\bfsigma = \{\sigma_1, \ldots, \sigma_k \} $ of distinct shifts that is closed under complex conjugation, let $\bfZ \in \mathbb{C}^{n \times q}$ be the basis resulting from the ADI iteration \eqref{eq:LRADI},
	and $\bfV \in \mathbb{R}^{n \times q}$ be an arbitrary real basis of the rational input block Krylov subspace $\mathcal{K}(\bfA,\bfB,\bfsigma)$.
	Let $\bfP_q \in \mathbb{R}^{q \times q}$ solve the projected Lyapunov equation
	\begin{equation}\label{eq:red_lyap_Thm}
		\bfA_q\bfP_q\bfE_q^T + \bfE_q\bfP_q\bfA_q^T +\bfB_q\bfB_q^T = \bfnull,
	\end{equation}
	with $\bfA_q := \bfW^T\bfA\bfV$, $\bfE_q := \bfW^T\bfE\bfV$ and $\bfB_q := \bfW^T\bfB$. 
	Then, there exists a matrix $\bfW \in \mathbb{R}^{n \times q}$, such that $\Ph := \bfV\bfP_q\bfV^T = \bfZ\bfZ^*$, and the associated explicit form  $\bfE_q^{-1}\bfA_q$, $\bfE_q^{-1}\bfB_q$ of the reduced system is unique.
\end{theorem}
\begin{proof}
	Assuming distinct shifts $\sigma_i\ne\sigma_j$, $i\ne j$, we have to show due to Theorem~\ref{thm:gugercin}, that there exists a $\bfW$ such that $\bfE_q^{-1}\bfA_q$ is diagonalizable with $k$ distinct eigenvalues $\bflambda = \{-\sigma_1, \ldots, -\sigma_k \}$, and each eigenvalue $\lambda_i = -\sigma_i$ has multiplicity $m$.
	Towards this aim we use the parametrization of all possible projected system matrices from \cite{Astolfi_2010,Wolf_2012_MathMod}:
	\begin{equation}\label{eq:S3}
		\bfE_q^{-1}\bfA_q = \bfS + \bfE_q^{-1}\bfB_q\bfL,
	\end{equation}
	where $\bfS \in \mathbb{R}^{q \times q}$ and $\bfL \in \mathbb{R}^{m \times q}$ are fixed for a given basis $\bfV$.
	Because $\bfV$ is a basis of a block Krylov subspace, the pair $(\bfS,\bfL)$ is observable and $\bfS$ is diagonalizable with the eigenvalues $\bfsigma = \{\sigma_1, \ldots, \sigma_k \} $, each with multiplicity $m$ \cite{Gallivan_2004,Wolf_2012_MathMod}.
	Now consider \eqref{eq:S3} as a pole-placement problem: we are searching for the ``feedback'' $\bfE_q^{-1}\bfB_q$, such that the eigenvalues of $\bfS$ are mirrored along the imaginary axis.
	Because the pair $(\bfS,\bfL)$ is observable, there exists a feedback $\bfE_q^{-1}\bfB_q$, that places all eigenvalues at the desired location.
	Due to multiplicity $m$ of the eigenvalues and \cite[Corollary $8$]{reilly_dof}, the desired feedback is unique, which shows that the desired matrices $\bfE_q^{-1}\bfA_q$ and $\bfE_q^{-1}\bfB_q$, such that $\bfV\bfP_q\bfV^T = \bfZ\bfZ^*$, are unique.
	It is left to show that there exists a $\bfW$, such that $\bfE_q^{-1}\bfB_q = (\bfW^T\bfE\bfV)^{-1}\bfW^T\bfB$ becomes the desired feedback, which is equivalent to $\bfW^T (\bfB - \bfE\bfV\bfE_q^{-1}\bfB_q ) = \bfnull$.
	Therefore, it is sufficient to show existence of a $\bfW \in \mathbb{R}^{n \times q}$ with $\colsp(\bfW) \subset \colsp[\bfE\bfV\; \bfB]$.
	By defining $\bfW := [\bfE\bfV\; \bfB]\bfK$ with $\bfK \in \mathbb{R}^{(q+m) \times q}$, this reads as $\bfK^T\bfM = \bfnull$, with $\bfM = [\bfE\bfV\; \bfB]^T (\bfB - \bfE\bfV\bfE_q^{-1}\bfB_q ) \in \mathbb{R}^{(q+m)\times m}$.
	That means, we are searching for a $q$-dimensional subspace that is orthogonal to an $m$-dimensional subspace in a $(q+m)$-dimensional space, which obviously exists.
	\hfill
\end{proof}

\begin{rem}
	This theorem---and thus also Theorem~\ref{thm:gugercin}---can be directly generalized to multiple shifts in the set $\bfsigma \!=\! \{\sigma_1, \ldots, \sigma_k \} $.
	Then the proof would basically not change, because one would have to show that there exists a $\bfE_q^{-1}\bfA_q$ with the eigenvalues $\lambda_i = -\sigma_i$, $i=1,\ldots,k$, each with geometric multiplicity $m$, and that the Jordan blocks to each eigenvalue have equal dimension.
	One can show that $\bfS$ in	\eqref{eq:S3} fulfills this, and with the same argument as above the result follows.
	The details of this proof, however, are omitted for a concise presentation.
	This generalization is of importance, because in a typical setting, one cyclically reuses an a priori chosen set of shifts in the ADI iteration, leading to multiple shifts in the set $\bfsigma$.
\end{rem}

	Theorem~\ref{thm:main} generalizes the results of \cite{Druskin_Convergence,flagg_ADI} in the following way:
	Instead of being restricted to particular sets of shifts $\bfsigma$ that fulfill the condition in Theorem~\ref{thm:gugercin}, the equivalence of ADI and RKSM can always be enforced for arbitrary sets $\bfsigma$, by using oblique projections in RKSM.
	This shows that the ADI iteration implicitly solves a particular projected Lyapunov equation---irrespective of the choice of shifts.
	This means that the approximation $\Ph = \bfZ\bfZ^*$ of the ADI iteration can be alternatively computed based on projections:
	once given a basis $\bfV$ of the Krylov subspace $\mathcal{K}(\bfA,\bfB,\bfsigma)$---this could also be the basis $\bfZ$ of the ADI iteration---the original matrices would have to be projected using an appropriate matrix $\bfW$, and the resulting reduced Lyapunov equation \eqref{eq:S3} then would have to be solved by direct methods.
	
	A possible way to compute a suitable $\bfW$ (the desired $\bfW$ is not unique), is to employ the pole-placement approach in \cite{antoulas_pp}; which, however, would require comparable numerical effort to the calculation of the basis $\bfV$ of the Krylov subspace.
	
	To avoid this, it is also possible to compute the desired reduced matrices $\bfP_q$, $\bfA_q$, $\bfE_q$ and $\bfB_q$ for a given $\bfV$ directly---without explicitly setting up $\bfW$.
	This is done by the \emph{pseudo-optimal rational Krylov} (PORK) algorithm in \cite{Wolf_2013_ECC}.
	It was presented for single inputs $\bfB \in \mathbb{R}^n$ only, but it can also be used for multiple inputs $\bfB \in \mathbb{R}^{n \times m}$ without modifications.
			
	We do not advocate to use the PORK algorithm for computing the approximation of the ADI iteration, due to higher numerical effort.
	However, it provides an interesting link between two different approaches for approximating the solutions of Lyapunov equations.
	This work provides the proves of this link between ADI and RKSM, which was first presented in the talk \cite{wolf_ADI_anif}.
	Furthermore, the link was already used in \cite{Wolf_at}, where the effect of the approximations $\Ph$ from ADI and RKSM on the reduced order model by approximate balanced truncation was investigated.

\subsection{$\H2$ pseudo-optimality of the ADI iteration}\label{sec:ADIH2}
A common way to measure the error in model order reduction is the $\H2$ norm, which is defined for a system \eqref{eq:sys} as
\begin{equation}
	\| \bfG \|_{\H2}^2 := \frac{1}{2\pi} \int_{-\infty}^{+\infty} \trace \left(      \overline{\bfG(j\omega)}\bfG(j\omega)^T \right) d\omega.
\end{equation}
It was shown in \cite{flagg_ADI} that the reduced system associated with the ADI iteration in Theorem~\ref{thm:gugercin} fulfills a so-called $\H2$ pseudo-optimality condition.
However, this pseudo-optimal\-ity is stated only for single inputs $m=1$, and it ``proves harder to extend'' to multiple inputs $m>1$, which is considered as an ``interesting research direction'' in \cite{flagg_ADI}. 
The following theorem identifies the general optimality of the ADI iteration in the sense of the $\H2$ norm.
To the best of the authors' knowledge, this is also the first attempt to generalize $\H2$ pseudo-optimality to block Krylov subspaces.
\begin{theorem}\label{thm:opt}
	Let $\bfC \in \mathbb{R}^{p \times n}$ be an arbitrary output matrix and define the reduced output by $\bfC_q := \bfC\bfV$.
	If the reduced system $\bfG_q(s) = \bfC_q \left( s\bfE_q - \bfA_q \right)^{-1} \bfB_q$ fulfills the conditions of Theorem~\ref{thm:main}, then it is an $\H2$ pseudo-optimal approximation of $\bfG(s)$, i.\,e.~it solves the following minimization problem:
	\begin{equation}\label{eq:min}
		\| \bfG- \bfG_q \|_{\H2} = \underset{\Grs\in \mathcal{T}_{(\bfA_q,\bfE_q)}^{(p,m)}}{\min}
		\| \bfG- \Grs \|_{\H2},
	\end{equation}
	where $\mathcal{T}_{(\bfA_q,\bfE_q)}^{(p,m)}$ is the set of all dynamical systems $\Grs(s) = \Crs \left( s\Ers - \Ars \right)^{-1} \Brs$,
	for which $\Ers^{-1}\Ars \in \mathbb{R}^{q \times q}$ and $\bfE_q^{-1}\bfA_q$ share the same Jordan canonical form, and $\Brs \in \mathbb{R}^{q \times m}$ and $\Crs \in \mathbb{R}^{p \times q}$ are arbitrary.
\end{theorem}

The proof can be found in Appendix~\ref{ap:proof2}.
\begin{rem}
	This theorem shows that computing the reduced system matrices $\bfA_q$, $\bfE_q$, $\bfB_q$ and $\bfC_q$ associated with the ADI iteration (e.\,g. by the PORK algorithm \cite{Wolf_2013_ECC}) yields an $\H2$ pseudo-optimal approximation of the original system \eqref{eq:sys}---irrespective of the choice of shifts and output matrix $\bfC$.
\end{rem}

\subsection{The residual of ADI}
For a given approximate solution $\Ph$, the residual in the Lyapunov equation \eqref{eq:lyap} is defined as
\begin{equation}\label{eq:res}
	\bfR := \bfA \Ph \bfE^T + \bfE \Ph \bfA^T + \bfB\bfB^T.
\end{equation}
It was shown in \cite{flagg_ADI}, that the residual $\bfR$ in the ADI iteration is orthogonal to the Krylov subspace, $\bfR\bfV = \bfnull$, if and only if the conditions of Theorem~\ref{thm:gugercin} are met.
In \cite{Druskin_Convergence}, additionally an explicit formulation of the residual is given, which, however, is inappropriate for numerical computations.
In the following, we present a new explicit formulation of the ADI residual, which is well-suited for numerical computations, easy to implement, and directly includes the above statement on orthogonality.
This formulation was first presented in the talk \cite{wolf_ADI_anif}, and then reworked by the authors of \cite{paed_MCMDS} with a different proof.
\begin{theorem}\label{thm:res}
	Let $\alpha_i := \sqrt{2\,\textrm{Re}(\sigma_{i})}$ and $\bfI$ denote the identity matrix of dimension $m \times m$.
	Then the residual $\bfR$ for the approximation $\Ph = \bfZ\bfZ^*$ with the basis $\bfZ = [\bfZ_1, \, \ldots, \, \bfZ_k ] \in \mathbb{C}^{n \times q}$ of the ADI iteration \eqref{eq:LRADI} is given by
	\begin{equation}
		\bfR = \Bs\Bss,
	\end{equation}
	where $\Bs = \bfB + \bfE\bfZ\bfL_{\text{ADI}}^T$ and $\bfL_{\text{ADI}} := \left[\alpha_1\bfI, \, \ldots,\, \alpha_k\bfI\right]$.
\end{theorem}
\begin{proof}
	The residual is given by
	\begin{equation}\label{eq:res2}
		\bfR = \bfA \bfZ\bfZ^* \bfE^T + \bfE \bfZ\bfZ^* \bfA^T + \bfB\bfB^T.
	\end{equation}
	Substituting $\bfA\bfZ$ with the Sylvester equation \eqref{eq:ADISyl} yields
	\begin{equation}\label{eq:res3}
		\bfR = \left(  \bfE\bfZ \bfS_{\text{ADI}} + \bfB \bfL_{\text{ADI}} \right) \bfZ^* \bfE^T + 
		\bfE \bfZ\left(  \bfE\bfZ \bfS_{\text{ADI}} + \bfB \bfL_{\text{ADI}} \right)^* + \bfB\bfB^T,
	\end{equation}
	which can be verified to be equivalent to
	\begin{equation}\label{eq:res4}
		\bfR = \left[ \bfE\bfZ, \, \bfB \right]
		\left[ \begin{array}{cc} \bfS_{\text{ADI}}+\bfS_{\text{ADI}}^T & \bfL_{\text{ADI}}^T \\ \bfL_{\text{ADI}} & \bfI_m \end{array}\right]
		\left[ \begin{array}{c} \bfZ^*\bfE^T \\ \bfB^T \end{array}\right].
	\end{equation}	
	It follows from \eqref{eq:S_ADI}, that $ \bfS_{\text{ADI}}+\bfS_{\text{ADI}}^T = \bfL_{\text{ADI}}^T\bfL_{\text{ADI}}$.
	Therefore,
	\begin{equation}\label{eq:res5}
		\bfR = \left[ \bfE\bfZ, \, \bfB \right]
		\left[ \begin{array}{c} \bfL_{\text{ADI}}^T \\ \bfI_m \end{array}\right]
		\left[  \bfL_{\text{ADI}}, \, \bfI_m \right]
		\left[ \begin{array}{c} \bfZ^*\bfE^T \\ \bfB^T \end{array}\right] = \Bs\Bss,
	\end{equation}		
	which completes the proof.
	\hfill
\end{proof}
\begin{rem}
	Although the basis $\bfZ \in \mathbb{C}^{n \times q}$ is complex for complex shifts, direct computation shows that the residual factor $\Bs = \bfB + \bfE\bfZ\bfL_{\text{ADI}}^T \in \mathbb{R}^{n\times m}$ is real, if in the sequence $\left(\sigma_1, \ldots, \sigma_k \right)$ each complex valued shift $\sigma_i$ is used as often as its complex conjugate.
\end{rem}
\begin{rem}
	The notation $\Bs$ stems from the fact, that the columns of $\Bs$ close the vector chain from the columns of $\bfB$ to their respective projections onto $\bfE\bfV$.
	This means that $\Bs$ is orthogonal to $\colsp(\bfW)$, which defines the direction of projection.
	Therefore, the residual always fulfills a Petrov-Galerkin condition $\bfR\bfW=\bfnull$, and the Galerkin condition $\bfR\bfV=\bfnull$ is met if and only if the conditions of Theorem~\ref{thm:gugercin} hold.
	This shows that the orthogonality conditions of \cite{Druskin_Convergence,flagg_ADI} are directly included in Theorem~\ref{thm:res}.
\end{rem}
\begin{corollary}\label{cor:R_iter}
	Define $\bfB_{\bot,0} := \bfB$ and let the ADI iteration \eqref{eq:LRADI} be augmented by the following iteration
	\begin{equation}\label{eq:Bs_iter}
		\bfB_{\bot,i} = \bfB_{\bot,i-1} + \sqrt{2\textrm{Re}(\sigma_i)} \bfE\bfZ_i, \qquad i = 1,2,\ldots,l.
	\end{equation}
	At an arbitrary step $1 \le j \le l$ in the ADI iteration, the Lyapunov approximation is given by $\Ph = \bfZ_j\bfZ_j^*$ with $\bfZ_j = [\bfZ_1, \, \ldots, \, \bfZ_j ]$.
	Then, the residual \eqref{eq:res} is given by $\bfR =  \bfB_{\bot,j}\bfB_{\bot,j}^T$.
\end{corollary}

The corollary directly follows from Theorem~\ref{thm:res} and shows that the formulation of the residual is well-suited for the iterative ADI procedure.
It further shows that the rank of the residual $\bfR$ is independent of the dimension $q$ of the ADI basis $\bfZ$: $\Rank(\bfR)\le m$.
The norm of the residual is often used as a convergence criterion in the ADI iteration.
A typical implementation is to approximate the Euclidean norm $\| \bfR \|_2$ via a power method, see \cite{diss_sabino}.
The new formulation here allows a fast computation of the Euclidean norm $\| \bfR \|_2$ by an $m\times m$ matrix:
$\| \bfR \|_2 \!=\! \| \Bs\Bss \|_2 \!=\! \| \Bss\Bs \|_2$.
For small $m$, the residual norm $\| \bfR \|_2$ can be calculated with negligible numerical effort---compared to the computation of the ADI basis $\bfZ$.
Therefore, with the new formulation from Corollary~\ref{cor:R_iter}, the norm $\| \bfR \|_2$ provides a fast-to-evaluate convergence criterion for the ADI iteration.

\subsection{An estimator of optimality of shifts}
It follows from Theorem~\ref{thm:main}, that the ADI iteration is generally associated with an oblique projection.
In contrast, RKSM is usually employed with an orthogonal projection---at least in the available literature.
It is interesting to investigate the case, when the ADI iteration gets related to an orthogonal projection, i.\,e. when both ADI and RKSM with $\bfW = \bfV$ yield the same approximate solution of the Lyapunov equation.

On the one hand, orthogonal projections are more favorable than oblique ones, due to better numerical behavior and advantages in stability preservation.
On the other hand, we showed that the oblique projection, that the ADI iteration is associated with, always fulfills a certain $\H2$ (pseudo-) optimality.
Hence, if the ADI iteration can be characterized by an orthogonal projection, both advantages are combined, 
and as shown in \cite{Breiten_ADI}, the error in ``the naturally induced energy norm of the corresponding linear operator of the Lyapunov equation'' is minimized in this case for symmetric systems.

Therefore, a set of shifts $\bfsigma$, such that ADI is associated with an orthogonal projection, can be considered optimal in some sense.
However, such a set is usually previously unknown and only an iterative algorithm can be stated (hopefully converging to such a set), which is often not computationally practical for large-scale systems.

With the results of this work, we cannot give a better algorithm to compute such a set, but at least we can give an a posteriori measure of the ``obliqueness'' of the projection that the ADI iteration is associated with.
As shown in the end of the section, this measure can then be used as an estimator of the quality of approximation.

To derive the measure, consider the matrices $[\bfA_q,\bfE_q,\bfB_q] = \bfW^T[\bfA\bfV,\bfE\bfV,\bfB]$.
Because $\bfV$ spans a rational Krylov subspace, $\colsp(\bfA\bfV) \subseteq \colsp([\bfE\bfV,\bfB])$.
This shows that only the part of $\bfW$ in the subspace $\colsp([\bfE\bfV,\bfB])$ is relevant for projection.
To state a unique measure, we therefore have to restrict ourselves to this subspace in the following, and choose $\bfW$ in the $(q+m)$-dimensional subspace given by $\colsp([\bfE\bfV,\bfB])$.
By decomposing the subspace $\colsp([\bfE\bfV,\bfB])$ into $\colsp(\bfW)$ and its orthogonal complement $\mathcal{W}^{\bot} := \{ \bfy \in \colsp([\bfE\bfV,\bfB]): \bfy^T\bfz = 0, \forall \bfz \in \colsp(\bfW) \}$, we find that $\colsp(\Bs) \subseteq \mathcal{W}^{\bot}$.
The orthogonal complement of $\colsp(\bfE\bfV)$ is defined as $\mathcal{V}^{\bot} := \{ \bfy \in \colsp([\bfE\bfV,\bfB]): \bfy^T\bfz = 0, \forall \bfz \in \colsp(\bfE\bfV) \}$, and a basis of this subspace can be computed as $\Bsv := \bfB - \bfE\bfV(\bfV^T\bfE\bfV)^{-1}\bfV^T\bfB$.
Now the angle $\theta$ between the orthogonal projection by $\mathcal{V} := \colsp(\bfV)$ and the oblique projection by $\mathcal{W} := \colsp(\bfW)$ onto $\colsp(\bfE\bfV)$ is given by the angle between the subspaces $\mathcal{V}^{\bot}$ and $\mathcal{W}^{\bot}$, or equivalently: the angle between the subspaces $\colsp(\Bsv)$ and $\colsp(\Bs)$.

The angle $\theta$ between the subspaces spanned by two matrices $\Bsv$ and $\Bs$ can be easily computed, e.\,g. in \textsc{MATLAB} with the command \texttt{subspace}.
Although this command is implemented for dense matrices, it can be easily implemented to also work for sparse matrices.
Assume that the ADI iteration \eqref{eq:LRADI} is implemented in \textsc{MATLAB} to compute the basis $\bfZ$, together with the computation of $\Bs$ by \eqref{eq:Bs_iter} (denoted as ``\texttt{Bp}'') for a given set of shifts $\bfsigma$.
Then a possible implementation for computing $\Bsv$ (denoted as ``\texttt{Bp\_EV}'') and the angle $\theta$ is:

\quad \parbox{0.9\columnwidth}
{\texttt{%
Bp\_EV = E*Z*( (Z'*E*Z)$\backslash$(Z'*B) );  \\
theta  = subspace(Bp\_EV,Bp);
}}

The smaller $\theta$ is, the closer the set $\bfsigma$ is to an $\H2$ pseudo-optimal set.
Please note, that this measure is not directly related to the approximation error $\|\bfP- \Ph \|$:
if $\theta$ is close to zero, one can expect $\Ph$ to be a good approximation of $\bfP$ for a certain rank of $\Ph$;
if $\theta$ is large, say close to $\pi\slash 2$, one cannot conclude that $\Ph$ is a bad approximation.
Especially in the typical setting, where a predetermined set of shifts $\bfsigma$ is cyclically reused until convergence occurs, it is very likely that $\theta$ is large.

To demonstrate this, we consider a short numerical example:
a semi-discretized heat transfer problem for optimal cooling of steel profiles from the Oberwolfach model reduction benchmark collection\footnote{Available at http://portal.uni-freiburg.de/imteksimulation/downloads/benchmark.}.
The order is $n\!\!=\!\!1,\!357$ so that $\bfP$ can be computed by direct methods for comparison.
We consider only the first input: $m\!\!=\!\!1$.

In order to find a set $\bfsigma$ that fulfills the conditions of Theorem~\ref{thm:gugercin} we used the \emph{iterative rational Krylov algorithm} (IRKA) \cite{H2_gugercin} in its one-sided version, i.\,e. with $\bfW\!=\!\bfV$.
We set the reduced order to $q\!=\!4$ and chose an initial set $\bfsigma \!=\! \{ 100,100,100,100 \}$.
After every iteration of IRKA, we computed the ADI basis $\bfZ$ for the resulting set of shifts $\bfsigma$, and also $\theta$ as proposed above.
Figure~\ref{fig:theta_over_IRKA} shows that $\theta$ tends to zero, which shows that IRKA indeed converges to a set $\bfsigma$ such that ADI is associated with an orthogonal projection.
The convergence of IRKA can also be concluded from Figure~\ref{fig:rel_norm_P_Ph_over_IRKA}, which shows that the relative error $\|\mathbf{P}\!-\! \widehat{\mathbf{P}} \|_2 / \|\mathbf{P} \|_2$ converges to a constant value.
\newlength{\maxwidth}
\setlength{\maxwidth}{0.48\columnwidth}
\newlength{\myheight}
\setlength{\myheight}{0.26\columnwidth}
\begin{figure}[!htb]
	\footnotesize
	\centering
	\subfigure[]{
	\includegraphics[width=\maxwidth,height=\myheight]{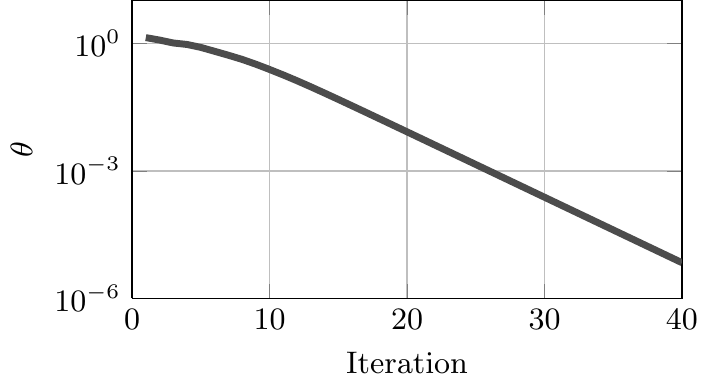}
	\label{fig:theta_over_IRKA}
	}\hfill
	\subfigure[]{
	\includegraphics[width=\maxwidth,height=\myheight]{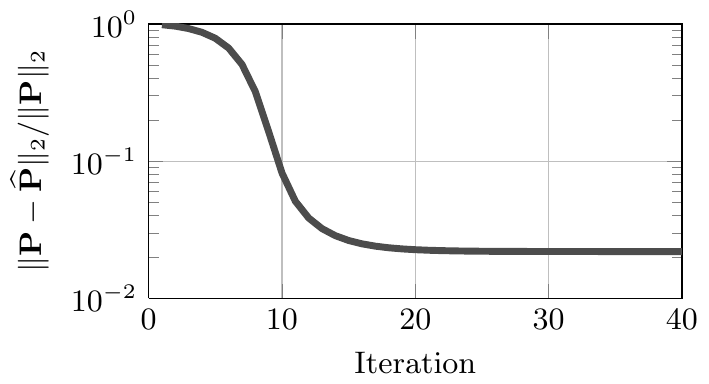}
	\label{fig:rel_norm_P_Ph_over_IRKA}
	}
	\caption{Convergence of IRKA to an optimal set $\bfsigma$, where  $n = 1357$, $q = 4$.}
\end{figure}

Although IRKA converges, which means that an $\H2$ pseudo-optimal set $\bfsigma$ is found, the approximation might not be sufficient.
This is due to the reduced order $q\!=\!4$, which is too small in this example to sufficiently approximate $\bfP$.
We therefore took the resulting set $\bfsigma \!=\! \{ \sigma_1,\ldots,\sigma_4 \}$ after $40$ iterations of IRKA and cyclically reused this set in the ADI iteration.
Figure~\ref{fig:rel_norm_P_Ph_over_cyclic_ADI} shows that the approximation error then tends to zero, and that $\bfP$ is approximated by $\Ph$ of rank $q\!=\!120$ with a relative error of $9.8 \cdot 10^{-10}$.
However, by reusing the set $\bfsigma$, it is not optimal anymore.
This can be concluded from Figure~\ref{fig:theta_over_cyclic_ADI}, which shows that $\theta$ rapidly tends to its maximum possible value $\pi/2$.
\begin{figure}[!htb]
	\footnotesize
	\centering
	\subfigure[]{
	\includegraphics[width=\maxwidth,height=\myheight]{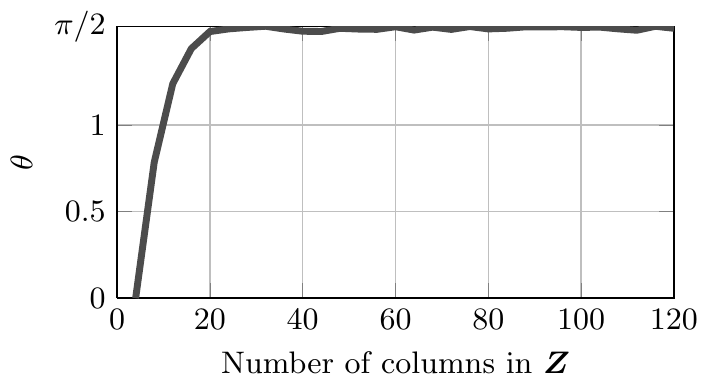}
	\label{fig:theta_over_cyclic_ADI}
	}\hfill
	\subfigure[]{
	\includegraphics[width=\maxwidth,height=\myheight]{theta_over_cyclic_ADI.pdf}
	\label{fig:rel_norm_P_Ph_over_cyclic_ADI}
	}
	\caption{ADI iteration for cyclically reusing the set $\bfsigma \!=\! \{ \sigma_1,\ldots,\sigma_4 \}$ from IRKA, where $n = 1357$.}
\end{figure}

This shows on the one hand, that $\theta$ close to zero indicates an optimal set for the respective reduced order, and on the other hand, that cyclically reusing shifts is far from optimal.
This, however, does not mean that the approximation has to be bad.
It only indicates, that for the respective reduced order a smaller error in $\bfP\!-\!\Ph$ should be possible; or equivalently, that the respective error in $\bfP\!-\!\Ph$ should also be reachable with a smaller reduced order.
Finding this better approximation, however, would require much higher numerical effort---and is a topic of current research.
To sum up, $\theta$ provides an estimator of the error $\bfP\!-\!\Ph$ for a given reduced order: it is sufficient but not necessary for a good approximation.

\section{Conclusions}
We have shown, that the ADI iteration for solving Lyapunov equations can always be interpreted as an oblique version of RKSM.
The results are based on a newly introduced Sylvester equation for the ADI basis, which facilitates the study of the ADI iteration.
The link to RKSM allows to associate a unique reduced order model to the ADI iteration which can be easily computed by the PORK algorithm.
The reduced model associated to the ADI iteration was proven to be an $\H2$ pseudo-optimal approximation of the original model with an arbitrary output.
This also generalizes previous results on $\H2$ pseudo-optimality to multivariable systems that are reduced by rational block Krylov subspaces.
Furthermore, a low-rank formulation of the Lyapunov residual is presented, which is well-suited for computation in the ADI iteration.
This allows to measure the ``obliqueness'' of the projection that the ADI iteration is related to---with negligible numerical effort.

\appendix


\section{Proof of Theorem~\ref{thm:opt}}\label{ap:proof2}
Let $\H2^{(p,m)}$ denote the set of all asymptotically stable systems \eqref{eq:sys} with $m$ inputs and $p$ outputs, which is a Hilbert space, \cite{Gug_H_approx}. 
The key to the proof is the observation that the set $\mathcal{T}_{(\bfA_q,\bfE_q)}^{(p,m)}$ is a closed subspace of $\H2^{(p,m)}$.
This follows from the fact that the sum of two systems in $\mathcal{T}_{(\bfA_q,\bfE_q)}^{(p,m)}$ stays in $\mathcal{T}_{(\bfA_q,\bfE_q)}^{(p,m)}$.
Assume for the moment one real shift $\sigma_1 \in \mathbb{R}$ which is reused $m_1$ times.
Due to the PORK algorithm in \cite{Wolf_2013_ECC}, the Jordan canonical form of $\bfE_q^{-1}\bfA_q \in \mathbb{R}^{(m_1 m) \times (m_1 m)}$ that fulfills the conditions of Theorem~\ref{thm:main} can be written with $-\sigma_1\bfI$ on the diagonal and $\bfI$ on the upper diagonal, where $\bfI$ denotes the identity matrix of dimension $m \times m$.
Then, without loss of generality, any system $\Grs(s)$ in $\mathcal{T}_{(\bfA_q,\bfE_q)}^{(p,m)}$ can be written as
\begin{align}
	\Grs(s) &:= \Crs \left(s\bfI- \bfE_q^{-1}\bfA_q \right)^{-1} \Brs = \\
	&= \Crs \left[\begin{array}{cccc}
	\frac{1}{s+\sigma_1}\bfI & \frac{1}{(s+\sigma_1)^2}\bfI & \cdots & \frac{1}{(s+\sigma_1)^{m_1}}\bfI \\ 
	& \ddots & \ddots & \vdots \\ & & \ddots & \frac{1}{(s+\sigma_1)^2}\bfI \\ & & & \frac{1}{s+\sigma_1}\bfI 
	\end{array} \right]\Brs = \sum_{i=1}^{m_1} \frac{\Theta_i}{(s+\sigma_1)^i},\label{eq:H(s)}
\end{align}
with the residuals $\Theta_i \in \mathbb{R}^{q \times m}$,
\begin{equation}\label{eq:Theta}
	\Theta_i = \sum_{j=1}^{m_1-i+1} \Csn_j \Bsn_{j+i-1}.
\end{equation}
where $\Brs := \left[\Bsn_1^T, \ldots, \Bsn_{m_1}^T \right]^T$ and $\Crs := \left[\Csn_1, \ldots, \Csn_{m_1} \right]$ are arbitrary.
The maximum possible rank of $\Theta_i$ is $\min(q,m)$.
Due to $\Bsn_{i} \in \mathbb{R}^{m \times m}$ and $\Csn_{i} \in \mathbb{R}^{p \times m}$ are arbitrary, residuals $\Theta_i$ of maximum rank are included in $\mathcal{T}_{(\bfA_q,\bfE_q)}^{(p,m)}$.
Therefore, the sum of two systems \eqref{eq:H(s)} stays in $\mathcal{T}_{(\bfA_q,\bfE_q)}^{(p,m)}$.	

Now assume a complex conjugated pair of shifts $\sigma_1, \bar{\sigma}_1 \in \mathbb{C}$ that are both used $m_1$ times.
Then $\Grs(s)$ becomes
\begin{equation}\label{eq:H(s)_c}
	\Grs(s) = \sum_{i=1}^{m_1} \frac{\Theta_i}{(s+\sigma_1)^i} +
 \sum_{i=1}^{m_1} \frac{\bar{\Theta}_i}{(s+\bar{\sigma}_1)^i},
\end{equation}
with complex residuals $\Theta_i \in \mathbb{C}^{q \times m}$.
Here again, the residuals $\Theta_i$ of maximum rank are included in $\mathcal{T}_{(\bfA_q,\bfE_q)}^{(p,m)}$; additionally, in the sum of two systems \eqref{eq:H(s)_c} the residuals stay complex conjugated, i.\,e. the sum of two systems \eqref{eq:H(s)_c} stay in $\mathcal{T}_{(\bfA_q,\bfE_q)}^{(p,m)}$.

If we assume arbitrary sets $\bfsigma$ with complex conjugated shifts of equal multiplicity, the different eigenvalues $-\sigma_i$ are decoupled in the Jordan canonical form of $\bfE_q^{-1}\bfA_q$.
Therefore, the above conclusions follow for each $\sigma_i$, which proves that $\mathcal{T}_{(\bfA_q,\bfE_q)}^{(p,m)}$ is a subspace.

Because $\mathcal{T}_{(\bfA_q,\bfE_q)}^{(p,m)}$ is a closed subspace of $\H2^{(p,m)}$, we can apply the Hilbert projection theorem to prove Theorem~\ref{thm:opt}.
With the $\H2$-inner product, defined for two systems $\bfG(s)$ and $\bfH(s)$ in $\H2^{(p,m)}$ as
\begin{equation}
	\left\langle \bfG, \bfH \right\rangle_{\H2^{(p,m)}} :=
	\frac{1}{2\pi} \int_{-\infty}^{\infty} \trace \left( \overline{\bfG(j\omega)}\bfH(j\omega)^T \right) d\omega,
\end{equation}
the Hilbert projection theorem states that $\bfG_q(s)$ is the minimizer of \eqref{eq:min} if and only if for all $\Grs(s)$ from $\mathcal{T}_{(\bfA_q,\bfE_q)}^{(p,m)}$:
\begin{equation}\label{eq:orth}
	\left\langle \bfG-\bfG_q, \Grs \right\rangle_{\H2^{(p,m)}} =
	\left\langle \bfG, \Grs \right\rangle_{\H2^{(p,m)}} -
	\left\langle \bfG_q, \Grs \right\rangle_{\H2^{(p,m)}} = 0
\end{equation}
The $\H2^{(p,m)}$-inner products in \eqref{eq:orth} can be computed by
\begin{align}
	\left\langle \bfG, \Grs \right\rangle_{\H2^{(p,m)}} &= \trace\left(\bfC\bfX_1\Crs^T\right), \qquad \textrm{and} \\
	\left\langle \bfG_q, \Grs \right\rangle_{\H2^{(p,m)}} &= \trace\left(\bfC_q\bfX_2\Crs^T\right), 
\end{align}
where $\bfX_1$ and $\bfX_2$ are the unique solutions of 
\begin{align}
	\bfA\bfX_1\bfE_q^T + \bfE\bfX_1\bfA_q^T + \bfB\Brs^T &= \bfnull, \label{eq:X1} \\
	\bfA_q\bfX_2\bfE_q^T + \bfE_q\bfX_2\bfA_q^T + \bfB_q\Brs^T &= \bfnull, \label{eq:X2}
\end{align}
see \cite{H2_gugercin}.
Let $\bfX$ be the solution of the Sylvester equation
\begin{equation}
	\bfA\bfX\bfE_q^T + \bfE\bfX\bfA_q^T + \bfB\bfB_q^T = \bfnull. \label{eq:X}
\end{equation}
Due to the duality of Krylov subspaces and Sylvester equations \cite{Gallivan_2004,Wolf_2012_MathMod}, \eqref{eq:X} and \eqref{eq:X1} can be interpreted in such way, that their solutions $\bfX$ and $\bfX_1$  span the rational input block Krylov subspace $\mathcal{K}(\bfA,\bfB,\bfsigma)$.
As shown in the proof of Theorem~\ref{thm:main}, $\bfG_q(s)$ is controllable and therefore, $\Rank(\bfX)=q$.
For this reason, $\colsp(\bfX_1)$ has to be contained in $\colsp(\bfX)$: $\colsp(\bfX_1)\subseteq \colsp(\bfX)$.
Thus, there exists a matrix $\bfT \in \mathbb{R}^{q \times q}$ (which is singular if $\Grs(s)$ is not controllable), such that $\bfX_1 = \bfX\bfT$.
From the PORK algorithm, we know that $\bfX = \bfV\bfP_q$, and we can substitute $\bfX_1 = \bfX\bfT = \bfV\bfP_q\bfT$ in \eqref{eq:X1}.
From Theorem~\ref{thm:main}, we know there exists a matrix $\bfW$, such that $\bfW^T\bfA\bfV = \bfA_q$, $\bfW^T\bfE\bfV = \bfE_q$ and $\bfW^T\bfB = \bfB_q$.
Multiplying \eqref{eq:X1} with this $\bfW^T$ from the left yields
\begin{equation}\label{eq:PqT}
	\bfA_q\bfP_q\bfT\bfE_q^T + \bfE_q\bfP_q\bfT\bfA_q^T + \bfB_q\bfB_H^T = \bfnull.
\end{equation}
As the solutions of \eqref{eq:X2} and \eqref{eq:PqT} are unique, we can identify $\bfX_2 = \bfP_q\bfT$.
Using this for \eqref{eq:orth} leads to
\begin{align}
	\left\langle \bfG-\bfG_q, \Grs \right\rangle_{\H2^{(p,m)}} &= 
	\trace\left(\bfC\bfX_1\Crs^T\right) - \trace\left(\bfC_q\bfX_2\Crs^T\right) \\
	&= \trace\left(\bfC\bfV\bfP_q\bfT\Crs^T\right) - \trace\left(\bfC_q\bfP_q\bfT\Crs^T\right) \\
	&= \trace\left(\bfC_q\bfP_q\bfT\Crs^T - \bfC_q\bfP_q\bfT\Crs^T\right) \qquad\qquad= \bfnull,
\end{align}
which completes the proof.

\section{Acknowledgments}
The authors thank Prof. Serkan Gugercin for the fruitful discussion at the MODRED 2013 in Magdeburg.

\bibliographystyle{plainnat}
\bibliography{ref}

\end{document}